\documentclass[a4paper,11pt]{article}
\usepackage[utf8]{inputenc}  
\usepackage[T1]{fontenc}
\usepackage[english]{babel}
\usepackage{ifthen}
\usepackage{array}
\usepackage{layout}
\usepackage{tabularx}
\usepackage{supertabular}
\usepackage{setspace}
\usepackage{lmodern}
\usepackage{graphicx}
\usepackage{epstopdf}
\usepackage{epsfig}
\usepackage{enumerate}
\usepackage{amsfonts}
\usepackage{amsmath}
\usepackage{amssymb}
\usepackage{amsthm}
\usepackage{mathrsfs}
\usepackage{algorithm,algorithmic}
\usepackage{tikz}
\usepackage{float} 
\usepackage[toc,page]{appendix}
\usepackage{framed}
\usepackage{calc}
\usepackage{bm}

\newtheorem{Thm}{Theorem}[section]
\newtheorem{Lemme}[Thm]{Lemme}
\newtheorem{Prop}[Thm]{Proposition}

\newtheorem{remark}[Thm]{Remark}

\newenvironment{changemargin}[2]{\begin{list}{}{%
\setlength{\topsep}{0pt}%
\setlength{\leftmargin}{0pt}%
\setlength{\rightmargin}{0pt}%
\setlength{\listparindent}{\parindent}%
\setlength{\itemindent}{\parindent}%
\setlength{\parsep}{0pt plus 1pt}%
\addtolength{\leftmargin}{#1}%
\addtolength{\rightmargin}{#2}%
}\item }{\end{list}}

\newcommand\R{\mathbb{R}}

\newcommand\T{\mathbb{T}}
\newcommand\N{\mathbb{N}}
\newcommand\Z{\mathbb{Z}}

\newcommand\Ng[1]{\left| #1 \right|_{\gamma}}
\newcommand\Hg{H_{\gamma}(\T)}
\newcommand\Nb[1]{\left| #1 \right|_{\beta}}
\newcommand\Hb{H_{\beta}(\T)}

\title{Damping operators for the BBM equation}
\author{Pierre GARNIER}
\date{2013}

\begin{document}

\title{\Large {\bf Long-time behavior of solutions of a BBM equation with
    generalized damping} % \footnote{Preprint submitted \hfill Decembre 2013}
} 
 
\author{ {\bf J.-P. Chehab, P. Garnier$^*$, Y. Mammeri}\\ 
{\small Laboratoire Ami\'enois de Math\'ematique Fondamentale et Appliqu\'ee,} \\  {\small CNRS UMR 7352, Universit\'e 
de Picardie Jules Verne,}\\
{\small 80069 Amiens, France.}\\
{\small $^*$Corresponding author: pierre.garnier@u-picardie.fr}
}

\date{ }
\maketitle

\begin{minipage}[t]{11.5cm}
	{\footnotesize {\bf Abstract.} We study the long-time behavior of the solution of a damped BBM equation $u_t + u_x - u_{xxt} + uu_x + \mathscr{L}_{\gamma}(u) = 0$. The proposed dampings $\mathscr{L}_{\gamma}$ generalize standards ones, as parabolic ($\mathscr{L}_{\gamma}(u)=-\Delta u$) or weak damping ($\mathscr{L}_{\gamma}(u)=\gamma u$) and allows us to consider a greater range. After establish the local well-posedness in the energy space, we investigate some numerical properties.} 
\end{minipage} \\ 
\\      
     
\begin{minipage}[t]{10.5cm}{\footnotesize {\bf Keywords.} 
BBM equation, dispersion, dissipation.}
\end{minipage}\\

{\footnotesize {\bf MS Codes.} 35B40, 35Q53, 76B03, 76B15.}

\section{Introduction}
	%\tableofcontents 
  %          ~~\\      
% main text                                    
%----------------------- Introduction ---------------------------------
The  one-way propagation of small amplitude, long wavelength, gravity
waves in shallow water can be described by the  Korteweg-de Vries equation
(KdV) \cite{KdV} 
$$
u_t + u_x + u_{xxx} + u u_x = 0.
$$
or its regularized version, the Benjamin-Bona-Mahony (BBM) equation 
$$
u_t + u_x - u_{xxt} + u u_x = 0.
$$

Many works related to the damped KdV
equation can be found in the literature (\cite{AmiBonSch, CabRos,
  CheSad1, CheSad2, Ghi1,
  Ghi2, Gou,
  GouRos, OttSud, Ven1, Ven2, Tem} and references therein). These account for a wide variety of different results, such as regularizing effect of the damping, asymptotic behavior, existence of attractor or numerical computations. Few literatures are
concerned the damped BBM equation \cite{HayKaiNau, Wan}. This paper  consists in 
establishing  theoretical and numerical results for the solution of a
generalized damped BBM equation.

We introduce the following equation, called dBBM, for $x \in \T(0,L)$, $L>0$ and 
$t\in \R$
$$
u_t + u_x - u_{xxt} + u u_x + \mathscr{L}_\gamma (u) = 0,
$$
where the operator $\mathscr{L}_\gamma$ is defined by its Fourier
symbol  
$$
\widehat{\mathscr{L}_\gamma (u )} (k) := \gamma_k \hat u_k.
$$
Here $\hat u_k$ is the $k-$th Fourier coefficient of $u$ and
$(\gamma_k)_{k\in \Z}$ are positive real numbers
chosen such that
$$
\int_\T  u(x) \mathscr{L}_\gamma (u) d\mu(x) = \sum_{k\in \Z} \gamma_k
|\hat u_k|^2 \geq 0.
$$

Standard dampings are included choosing $\gamma_k = k^2$ as parabolic
damping or $\gamma_k = \gamma$ as weak damping. The proposed sequence
$(\gamma_k)_{k\in \Z}$  allows us to consider a
greater range of damping. In particular, one may wonder
if it is necessary to absorb all frequencies of a long wavelength gravity
wave like the solution of KdV or BBM. To obtain this kind of results, we can consider sequences $\gamma_k \rightarrow 0$ when $|k| \rightarrow +\infty$ or even $\gamma_k = 0$ for large $|k|$.

In comparison with the standard BBM equation for which the $H^1-$norm
is preserved, we notice that the $H^1$-norm decreases for the solution $u$ of the damped
equation dBBM. More precisely, we have for all time $t\in \R$ 
$$
\frac{d}{dt} ||u(t)||_{H^1}^2 = - |u(t)|_{\gamma}^2 \leq 0,
$$
and the natural space to study the well-posedness is the space
$H_\gamma(\T)$, defined as
$$
H_\gamma(\T) := \left\{u\in L^2(\T); \sum_{k\in \Z} \gamma_k 
|\hat u_k|^2 < +\infty \right\},
$$
equipped with the norm
$$
|u|_\gamma := \sqrt{ \sum_{k\in \Z} \gamma_k
|\hat u_k|^2}.
$$
In order to simplify the writings, $k$ can denote either an integer as an index, or the value $\frac{2\pi k}{L}$.
% For example instead of $\widehat{u_x}(k) = \frac{2i\pi k}{L} \hat{u}_k$, we will write $\widehat{u_x}(k) = ik \hat{u}_k$.
%In order to simplify the writtings, we take $L=2\pi$. Indeed, we have $k$ instead of $\frac{2\pi k}{L}$. And the results presented in section 2 and 3 are identical with other value of $L$.
When there is no ambiguity, we denote $C$ the different constants appearing in the following results. 

The paper is organised as follows. We establish in Section 2 the
local well-posedness of the dBBM equation. We prove some important
estimates about the space $H_\gamma(\T)$. Some qualitative properties
of the solution are also given. In Section 3, numerical schemes are
presented to solve the damped equation and to preserve the qualitative
properties of the continuous solution. Section 4 deals with the
numerical results. In particular, we manage to build a family of dampings
weaker than the standard ones. These, for example
taking $ \lim_{| k | \to +\infty} \gamma_k = $ 0, yet provide
dampen solutions.

\section{Analysis of the problem}
	After studying the space $\Hg$, we establish the well-posedness of the Cauchy problem associated with the dBBM equation.	
	
	\subsection{Proper energy space $\Hg$}
	Here we take $\gamma_k>0, \ \forall k \in \Z$. We first state some properties of injection.

\begin{Prop}
	Assume that $\sum\limits_{k \in \Z} \frac{1}{\gamma_k} < +\infty$. Then there exists a constant $C>0$ such that 
	\[ \| u \|_{\infty} \leq  C |u|_{\gamma}. \]
\end{Prop}

The injection $H_{\gamma}(\T) \hookrightarrow  L^{\infty}(\T)$ is continuous.

\begin{proof}
	Let $u \in H_{\gamma}(\T)$. We notice that 
	\[ u(x)=\sum\limits_{k \in \Z} \hat{u}(k) e^{ik x}. \]
	Then
	\[ |u(x)| \leq \sum\limits_{k \in \Z} |\hat{u}(k)| =  \sum\limits_{k \in \Z} \frac{1}{\sqrt{\gamma_k}} \sqrt{\gamma_k} |\hat{u}(k)|.\]
	We assumed that $\gamma_k > 0$. Hence, the Cauchy-Schwarz inequality implies for all $x \in \T$ :
	\[ |u(x)| \leq \left( \sum\limits_{k \in \Z} \frac{1}{\gamma_k}\right)^{\frac{1}{2}} \left( \sum\limits_{k \in \Z} \gamma_k |\hat{u}(k)|^2 \right)^{\frac{1}{2}} = \left( \sum\limits_{k \in \Z} \frac{1}{\gamma_k}\right)^{\frac{1}{2}} \Ng{u}. \]
	This completes the proof.
\end{proof}

\begin{remark}
The above result is also a condition to have $\| \hat{u} \|_{l^1} \leq C |u|_{\gamma}$.
\end{remark}

\begin{Prop}
	We assume that for all $k \in \Z$ we have $\gamma_k > \beta_k$. Let $\rho_N = \max\limits_{k \geq N} \frac{\beta_k}{\gamma_k}$. Then $\lim\limits_{N \rightarrow +\infty} \rho_N = 0$ if and only if the continuous injection $\Hg \hookrightarrow  \Hb$ is compact.
\end{Prop}

\begin{proof}
	The condition is necessary, indeed if there exists $\alpha>0$ such that $\rho_N > \alpha, \ \forall N$, then the norms $\Nb{u}$ and $\Ng{u}$ are equivalent, the injection cannot be compact.
	Let us prove now that the condition is sufficient.	
	First we have for $u \in \Hg$ :
	\[ \Nb{u} = \sum\limits_{k \in \Z} \beta_k |\hat{u}_k|^2 \leq \sum\limits_{k \in \Z} \gamma_k |\hat{u}_k|^2 = \Ng{u}.\]
	This shows that the injection is continuous. Now we prove that the injection is compact. We use finite rank operators and we take the limit. Let $I_N$ be the orthogonal operator on the polynomials of frequencies $k$ such that $-N \leq k \leq N$. We have
	\[ I_N u = \sum\limits_{|k|\leq N} \hat{u}_k e^{i kx}. \]
	Hence
	\begin{align*}
		\Nb{(Id-I_N)u}^2 & = \sum\limits_{|k| \geq N+1} \beta_k |\hat{u}_k|^2, \\
			 & \leq \sum\limits_{|k| \geq N+1} \frac{\beta_k}{\gamma_k} \gamma_k |\hat{u}_k|^2, \\
			 & \leq \rho_{N+1} \Ng{u}^2 \underset{N \rightarrow + \infty}{\longrightarrow} 0.
	\end{align*}
	Therefore $Id$ is a compact operator and consequently the injection is compact.		 
\end{proof}

Now we give some conditions on the sequence $(\gamma_k)_{k \in \Z}$ so that $\Hg$ is an algebra.
\begin{Prop}
	Let $u$ and $v$ be two functions in $H_{\gamma}(\T)$. Assume that there exists a constant $C>0$ such that $\sqrt{\gamma_k} \leq C \left( \sqrt{\gamma_{k-j}} + \sqrt{\gamma_j} \right)$ for all $k,\ j \in \Z$. Then we have the following inequality:
	\[ |uv|_{\gamma} \leq C \left( |u|_{\gamma} \| \hat{v} \|_{l^1} + |v|_{\gamma} \| \hat{u} \|_{l^1}  \right), \]
	Moreover if $\sum\limits_{k \in \Z} \frac{1}{\gamma_k} < +\infty$ then $H_{\gamma}(\T)$ is an algebra.
\end{Prop}

\begin{proof}
	Let $u, \, v \in H_{\gamma}(\T)$. We have
	\[ |uv|_{\gamma}^2 = \sum\limits_{k \in \Z} \gamma_k |\widehat{uv}(k)|^2. \]
	We remind that $\widehat{uv}(k) = \hat{u} \star \hat{v} (k)$. We use the inequality 
	\[ \sqrt{\gamma_k} \leq C \left( \sqrt{\gamma_{k-j}} + \sqrt{\gamma_j} \right).\]
	We obtain for all $k \in \Z$
	\[ \sqrt{\gamma_k} |\widehat{uv}(k)| \leq C \left( \sum\limits_{j \in \Z} \sqrt{\gamma_{k-j}} |\hat{u}(k-j) \hat{v}(j)| + \sum\limits_{j \in \Z} \sqrt{\gamma_j} |\hat{u}(k-j) \hat{v}(j)| \right). \]
	Hence
	\begin{align*}
		|uv|_{\gamma}^2 & \leq C^2 \sum\limits_{k \in \Z} \left( \sum\limits_{j \in \Z} \sqrt{\gamma_{k-j}} |\hat{u}(k-j) \hat{v}(j)| + \sum\limits_{j \in \Z} \sqrt{\gamma_j} |\hat{u}(k-j) \hat{v}(j)| \right)^2, \\
			& \leq C^2 \sum\limits_{k \in \Z} \left[ \left( \sum\limits_{j \in \Z} \sqrt{\gamma_{k-j}} |\hat{u}(k-j) \hat{v}(j)| \right)^2 + \left( \sum\limits_{j \in \Z} \sqrt{\gamma_j} |\hat{u}(k-j) \hat{v}(j)| \right)^2 \right], \\
			& \leq C^2 \left( \| \left( \sqrt{\gamma_k}|\hat{u}| \right) \star |\hat{v}| \|_{l^2}^2 + \| |\hat{u}| \star \left( \sqrt{\gamma_k}|\hat{v}| \|_{l^2}^2 \right) \right).
	\end{align*}
	We remind that for two functions $f$ and $g$ defined from $\N$ to $\R$ such that $f \in l^1$ and $g \in l^2$, we have
	\[  \| |f| \star |g| \|_{l^2}^2 \leq \| g \|_{l^2}^2 \| f \|_{l^1}^2. \]		
	Thus
	\[ |uv|_{\gamma}^2 \leq C \left( |u|_{\gamma}^2 \|\hat{v}\|_{l^1}^2 + |v|_{\gamma}^2 \|\hat{u}\|_{l^1}^2 \right). \]
	We know there exists a constant $c>0$ such that $\| \hat{u} \|_{l^1} \leq c |u|_{\gamma}$ if $\sum\limits_{k \in \Z} \frac{1}{\gamma_k} < +\infty$. Then, there exists $\tilde{C}>0$ such that
	\[ |uv|_{\gamma} \leq \tilde{C} |u|_{\gamma} |v|_{\gamma}.\]
\end{proof}

\begin{remark}
	With $\gamma_k=k^{2s}$ we find the standard Sobolev injection. Indeed, we have the inequality $\sqrt{\gamma_k} \leq C \left( \sqrt{\gamma_{k-j}} + \sqrt{\gamma_j} \right)$ and if $s>\frac{1}{2}$ then $\sum\limits_{k \in \Z} \frac{1}{\gamma_k} < +\infty$.
\end{remark}
	
	\subsection{Local well-posedness}
	We can study the well-posedness of problem dBBM. We consider the Cauchy problem
	\begin{align}
		& u_t + u_x - u_{xxt} + uu_x + \mathscr{L}_{\gamma}(u) = 0, \quad x \in \T, \ t \in [0,T] \label{pb_cauchy1}\\
		& u(x,t=0) = u_0(x). \label{pb_cauchy2}
	\end{align}

\begin{Thm}\label{thm_local}
	Assume that $\sqrt{\gamma_{k+j}} \leq C(\sqrt{\gamma_k} + \sqrt{\gamma_j})$ for all $k$ and $j$ in $\Z$, $\sum\limits_{k \in \Z} \frac{1}{\gamma_k} < +\infty$ and $u_0 \in \Hg$. Then there exists $T=\frac{1}{C_0 \Ng{u_0}}>0$ and there exists a unique $u \in \mathscr{C} \left( [0,T],\Hg \right)$ solution of the Cauchy problem \eqref{pb_cauchy1}-\eqref{pb_cauchy2}.

	Moreover, for all $M>0$ with $\Ng{u_0} \leq M$ and $\Ng{v_0} \leq M$, there exists a constant $C_1>0$ such that the solutions $u$ and $v$, associated with the initial data $u_0$ and $v_0$ respectively, satisfy for all $t \leq \frac{1}{C_0 M}$
	\[ \Ng{u(\cdot,t) - v(\cdot,t)} \leq C_1 \Ng{u_0 - v_0}. \]
\end{Thm}

\begin{proof}
	Let write the initial value problem \eqref{pb_cauchy1}-\eqref{pb_cauchy2} as
	\begin{align}
		\partial_t u = Au + f(u), \label{nonlinpart1}\\
		u(t=0) = u_0(x), \label{nonlinpart2}
	\end{align}
	where
	\begin{align*}
		& Au = -\left( 1-\partial_{xx} \right)^{-1} \left( \partial_x u + \mathscr{L}_{\gamma} (u) \right),\\
		& f(u) = -\left( 1-\partial_{xx} \right)^{-1} \partial_x \left( \frac{u^2}{2} \right).\\
	\end{align*}
	According to the Duhamel's formula, $u$ is solution of the Cauchy problem \eqref{nonlinpart1}-\eqref{nonlinpart2} if and only if 
	\[ u(t) = \phi \left( u(t) \right) = S_t u_0 + \int_0^t S_{t-\tau} f\left( u(\tau) \right) d\tau, \]
	where	
	\[ S_t v = \sum\limits_{k \in \Z} e^{i kx} e^{-\frac{(ik+\gamma_k)}{1+k^2}t} \hat{v}(k). \]
	The purpose is to show that $u$ is the unique fixed point of $\phi$. We define the closed ball 
	\[ \overline{B}(T) = \left\{ u \in \mathscr{C} \left( [0,T];H_{\gamma}(\T) \right) ; |u(t)-u_0|_{\gamma} \leq 3 |u_0|_{\gamma} \right\}. \]
	Our purpose is to apply the Banach fixed point theorem.
	
	Let $u \in \overline{B}(T)$. We show that $\phi \left( u(t) \right) \in \overline{B}(T)$.
			  From the triangle inequality, we have:
			  \[ |\phi \left( u(t) \right)|_{\gamma} \leq |S_t u_0|_{\gamma} + \int_0^t |S_{t-\tau} \left( f(u) \right)|_{\gamma} d\tau. \]		
			  On the one hand, since $\gamma_k \geq 0$
			  \begin{align*}
			  	  |S_t u_0|_{\gamma}^2 & = \sum\limits_{k \in \Z} \gamma_k |\widehat{S_t u_0} (k)|^2, \\
			  	  	  & \leq \sum\limits_{k \in \Z} \gamma_k |e^{-\frac{\gamma_k}{1+k^2} t} \widehat{u_0} (k)|^2, \\
			  	  	  & \leq |u_0|_{\gamma}^2, 
			  \end{align*}
			  and
			  \[ |S_{t - \tau} f \left( u(\tau) \right)|_{\gamma} \leq |f \left( u(\tau) \right)|_{\gamma}. \] 
			  On the other hand, we need to upper bound the term $|f \left( u(\tau) \right)|_{\gamma}$. We have
			  \begin{align*}
			  	  |f \left( u(\tau) \right)|_{\gamma}^2 & = |- \left( 1 - \partial_{xx} \right)^{-1} \partial_x \left( \frac{u^2}{2}(\tau) \right)|_{\gamma}^2, \\
			  	  	  & = \sum\limits_{k \in \Z} \gamma_k \left| \mathscr{F} \left( - \left( 1 - \partial_{xx} \right)^{-1} \partial_x \left( \frac{u^2}{2}(\tau) \right) \right) \right|^2, \\
			  	  	  & = \sum\limits_{k \in \Z} \gamma_k \left| \frac{ik}{1+k^2} \widehat{\frac{u^2}{2}} \right|^2.
			  \end{align*}
			  However, since $\frac{|k|}{1+k^2} \leq 1$ and $\Hg$ is an algebra, it gets
			  \[ \Ng{f \left( u(\tau) \right)} \leq C \Ng{u}^2. \]
			  The Duhamel's formula implies for $0 \leq t \leq T$
			  \begin{align*}
			  	  |\phi \left( u(t) \right)|_{\gamma} & \leq |u_0|_{\gamma} + \int_{0}^{t} |u(\tau)|_{\gamma}^2 d\tau, \\
			  	  	  & \leq |u_0|_{\gamma} + C \left( \sup\limits_{t \in [0,T]} |u(t)|_{\gamma} \right)^2 T.
			  \end{align*}
			  But $u \in \overline{B}(T)$, so
			  	 \[ \Ng{u(t)}-\Ng{u_0} \leq |u(t)-u_0|_{\gamma} \leq 3|u_0|_{\gamma}, \]
			  that implies
			  	 \[ |u(t)|_{\gamma} \leq 4 |u_0|_{\gamma}. \]
			  We have $\phi \left( u(t) \right) \in \overline{B}(T)$ if $|\phi \left( u(t) \right) -u_0|_{\gamma} \leq 3 |u_0|_{\gamma}$. 
			  	 The inequality
			  	 \[ |\phi \left( u(t) - u_0\right)|_{\gamma} \leq 2|u_0|_{\gamma} + C T \left( 16 |u_0|_{\gamma}^2 \right) \leq 3|u_0|_{\gamma}, \]
			  is true if we choose
			  \[ 0 < T \leq \frac{ |u_0|_{\gamma}}{16C |u_0|_{\gamma}^2} = \frac{1}{16C} \left( \frac{1}{|u_0|_{\gamma}} \right). \]
			  
		Now we show that $\phi$ is a contraction mapping on $\overline{B}(T)$.
			  Let $u,\ v \in \overline{B}(T)$. We have
			  \begin{align*}
			   \Ng{\phi \left( u(t) \right) - \phi \left( v(t) \right)} & = \Ng{\int_0^t S_{t-\tau} \left(   f(u) - f(v) \right) d\tau}, \\
			   & \leq \int_0^t \Ng{\frac{u^2}{2} - \frac{v^2}{2}} (\tau) d\tau. 
			   \end{align*}
			  We notice that $u^2-v^2=(u-v)(u+v)$. Since $\Hg$ is an algebra, we have
			  \begin{align*}
			  	  \Ng{u^2-v^2} & \leq C \Ng{u-v}\Ng{u+v}, \\
			  	  	  & \leq C \left( \Ng{u} + \Ng{v} \right) \Ng{u-v}.
			  \end{align*}
			  And for $u,\ v \in \overline{B}(T)$,
			  \[ \Ng{u^2 - v^2} \leq 8 C \Ng{u_0} \Ng{u-v}. \]
			  We infer for all $t \in [0,T]$
			  \begin{align*}
			  	  \Ng{\phi \left( u(t) \right) - \phi \left( v(t) \right)} & \leq 8C \Ng{u_0} \int_0^t \Ng{u-v} (\tau) d\tau, \\
			  	  	  & \leq 8 C \Ng{u_0} T \sup\limits_{t \in [0,T]} \Ng{(u-v)(t)}.
			  \end{align*}
			  Hence
			  \[ \sup\limits_{t \in [0,T]} \Ng{\phi \left( u(t) \right) - \phi \left( v(t) \right)} \leq \left( 8 C \Ng{u_0}T \right) \sup\limits_{t \in [0,T]} \Ng{u(t)-v(t)}. \]
			  Consequently $\phi$ is a contraction mapping if $8C \Ng{u_0} T <1$, i.e., $T<\frac{1}{8C\Ng{u_0}}$.
			  Finally, from the Banach fixed-point theorem, $\phi$ has a unique fixed-point solution of $u(t)=\phi\left( u(t) \right)$. So, there exists a unique solution of the Cauchy problem.

			 It remains to prove the continuity with the initial data. Let $u$ and $v$ solutions of the Cauchy problem \eqref{pb_cauchy1}-\eqref{pb_cauchy2} with initial data $u_0$ and $v_0$ respectively, such that $\Ng{u_0} \leq M$ and $\Ng{v_0} \leq M$. The Duhamel's formula gives for $t\in [0,T]$, $T \leq \frac{1}{C_0 M}$
			 \begin{align*}
			 	\Ng{u-v} & \leq \Ng{u_0 - v_0} + \int_0^t \Ng{f(u) - f(v)} d\tau, \\
			 	& \leq \Ng{u_0 - v_0} + C'T \left( \Ng{u_0} + \Ng{v_0} \right) \sup\limits_{t \in [0,T]} \Ng{u-v}. \\
			 \end{align*}
			 It implies 
			 \[ \Ng{u-v} \leq C_1 \Ng{u_0 - v_0}. \]
\end{proof}

	\subsection{Behavior of the solution}
	In this part, we aim to get some estimations of the damping rate. Actually in the case without damping, the $H^1$-norm of the solution is invariant during time. Here this norm is decreasing. We adapt the work done on KdV equation \cite{CheSad2}.

Let us begin with the linear equation, that reads
\begin{equation}\label{lBBM}
	u_t-u_{txx}+u_x+\mathscr{L}_{\gamma}(u)=0 \quad x \in \T, \ t \in [0,T].
\end{equation}

Let $u$ be valued in $L^2(\T), \ \forall t>0$. We write $u$ as a Fourier series and, due to the orthogonality of the trigonometric polynomials, we obtain
\[ (1+k^2)\frac{d\hat{u}_k(t)}{dt} + (\gamma_k + ik) \hat{u}_k(t) = 0.\]
Hence
\[ \hat{u}_k(t)=e^{-\frac{\gamma_k + ik}{1+k^2}t} \hat{u}_k(0). \]
It follows that
\[ |u|_{H^1}^2 = \sum\limits_{k \in \mathbb{Z}} (1+k^2)|\hat{u}_k(t)|^2 = \sum\limits_{k \in \mathbb{Z}} (1+k^2) e^{-\frac{2\gamma_k}{1+k^2}t} |\hat{u}_k(0)|^2.\]

\begin{Prop}
	Let $\gamma_k > 0, \ \forall k \in \mathbb{Z}$ and $u_0 \in H_{\delta}(\T)$ where $\delta_k = \frac{(1+k^2)^2}{\gamma_k}$. Then the unique solution $u$ of \eqref{lBBM} verifies
	\[ |u|_{H^1}^2 \leq \min \left( \frac{e^{-1}}{2t}|u_0|_{\delta}^2 , |u_0|_{H^1}^2 \right), \ \forall t > 0. \]
	More generally, if $u_0 \in H_{\beta\delta}(\T)$ then
	\[ |u|_{\beta'}^{2} \leq \frac{e^{-1}}{2t}|u_0|_{\beta\delta}^2 \mbox{~where~} \beta_k'=(1+k^2)\beta_k. \]
\end{Prop}

\begin{proof}
	On one hand, the scalar product in $L^2(\T)$ of \eqref{lBBM} with $u$ provides 
	\[ \frac{1}{2}\frac{d}{dt}|u|_{H^1}^2 + \sum\limits_{k \in \mathbb{Z}} \gamma_k |\hat{u}_k(t)|^2=0.\]
	Hence 
	\[ \frac{d}{dt}|u|_{H^1}^{2} \leq 0. \]
	Consequently $|u|_{H^1}^{2} \leq |u_0|_{H^1}^{2}$.
	On the other hand, to obtain the second inequality, we write
	\begin{eqnarray*}
		|u|_{H^1}^2 & = & \sum\limits_{k \in \mathbb{Z}} (1+k^2) |\hat{u}_k|^2, \\
		 & = & \sum\limits_{k \in \mathbb{Z}} (1+k^2) e^{-\frac{2\gamma_k}{1+k^2}t} |\hat{u}_k(0)|^2, \\
		 & = & \sum\limits_{k \in \mathbb{Z}} (1+k^2)^2 \frac{\gamma_k}{1+k^2} e^{-\frac{2\gamma_k}{1+k^2}t} \frac{1}{\gamma_k}|\hat{u}_k(0)|^2.
	\end{eqnarray*}
	Since the function $\beta \mapsto \beta e^{-2\beta t}$ is uniformly bounded by $\frac{e^{-1}}{2t}$, we infer that 
	\[ |u|_{H^1}^2 \leq \frac{e^{-1}}{2t}|u_0|_{\delta}^2. \]
	In the same way we obtain the third inequality from
	\[ \beta_k |\hat{u}_k|^2 = \beta_k e^{-\frac{2\gamma_k}{1+k^2}t}|\hat{u}_k(0)|^2 = (1+k^2) \frac{\gamma_k}{1+k^2} e^{-\frac{2\gamma_k}{1+k^2}t} \left( \frac{\beta_k}{\gamma_k} |\hat{u}_k(0)|^2 \right).\]
\end{proof}

We can prove the two following results with similar arguments:

\begin{Prop}
	Assume that $\gamma_k \in [0,1], \ \forall k \in \mathbb{Z}$ and $u_0 \in H_{\delta^{(s)}}$, where $\delta_k^{(s)} = \frac{(1+k^2)^{s+1}}{\gamma_k^s}$. Then for all $s>0$, $u$ the solution of \eqref{lBBM} verifies $\forall t>0$
	\[ |u|_{H^1}^2 \leq \min \left( e^{-s}\left( \frac{s}{2t} \right)^s |u_0|_{\delta^{(s)}}^2, \ |u_0|_{H^1}^2 \right). \]
\end{Prop}

\begin{Prop}
	We assume that there exist three positive constants $\alpha$, $\beta$ and $C$ such that $u$ the solution of \eqref{lBBM} verifies:
	\begin{enumerate}[i.]
		\item $|\widehat{u_0}_k|^2 \leq C \gamma_k^{2\delta}$ with $\delta = \alpha + \beta$.
		\item $\sum\limits_{k \in \mathbb{Z}}(1+k^2)^{2\alpha+1}\gamma_k^{2\beta} < +\infty$.
	\end{enumerate}
	Then
	\[ |u|_{H^1}^2 \leq C e^{-2a}\left( \frac{\alpha}{t} \right)^{2\alpha} \sum\limits_{k \in \mathbb{Z}}(1+k^2)^{2\alpha+1}\gamma_k^{2\beta} = O\left( \frac{1}{t^{2\alpha}} \right).  \]
\end{Prop}

%-----------------------------------------------------------------------------------------%
%                              Equation non linéaire homogène                             %
%-----------------------------------------------------------------------------------------%
We can now deal with the non-linear equation. We can find similar kind of decreasing but less explicit than in the linear case.
We remind the equation
\begin{equation} \label{eqBBM2}
    u_t-u_{txx}+u_x+uu_x+\mathscr{L}(u)=0
\end{equation}
The scalar product in $L^2$ of (\ref{eqBBM2}) with $u$ yields
\[ \frac{1}{2} \frac{d}{dt} |u|_{H^1}^2 + |u|_{\gamma}^2 = 0. \]

\begin{Prop}\label{prop4}
	We assume that
	\begin{enumerate}[i.]
		\item $\gamma_k > 0 \ \forall k \in \mathbb{Z}$,
		\item $u_0 \in \Hg \cap H^1(\T)$ with $\int\limits_0^L u_0(x)dx=0$,	
	\end{enumerate}	  
	Then $\lim\limits_{t \rightarrow + \infty} |u|_{H^1} = 0$.
\end{Prop}

\begin{proof}
	%Commençons par démontrer la première affirmation.
	Because of $\gamma_k>0$, we have $\frac{d |u|_{H^1}^2}{dt} \leq 0$ and then the function $t \mapsto |u|_{H^1}^2$ is decreasing and consequently $\lim\limits_{t \rightarrow + \infty} |u(t)|_{H^1}^2 = C$. We notice that we also have $|u|_{H^1} \leq |u_0|_{H^1}$. Since $u \in H^1 \cap H_{\gamma}$, we deduce that $\lim\limits_{t \rightarrow + \infty} |u(t)|_{\gamma}^2 = 0$. Finally, since $\gamma_k > 0$, we have $\lim\limits_{t \rightarrow + \infty} \hat{u}_k = 0$ and then $u=0$, which means that $C=0$.
	
	%Maintenant pour la deuxième affirmation, 
\end{proof}

\begin{Lemme}
	For all function $v$ smooth enough, we set $\bar{v} = \frac{1}{L} \int\limits_{0}^{L} v(x,t) dx$. Let $u$ be the solution of dBBM valued in $\Hg \cap H^2(\T)$. Then $\bar{u}(t) = e^{-\gamma_0 t}\bar{u}(0)$.
\end{Lemme}

\begin{proof}
	We integrate the equation in space on the interval $[0,L]$ and we obtain
	\[ \int\limits_{0}^{L} \left( \frac{\partial u}{\partial t} + \mathscr{L}_\gamma (u) \right) dx = 0 .\]
	But
	\[ \int\limits_{0}^{L} \mathscr{L}_{\gamma} (u)(x,t) dx = \int\limits_{0}^{L} \sum\limits_{k \in \mathbb{Z}} \gamma_k \hat{u}_k e^{\frac{2i\pi kx}{L}} dx = L\gamma_0 \hat{u}_0 = \gamma_0 \int\limits_{0}^{L} u \ dx.\]
	Hence we have 
	\[ \frac{d}{dt} \left( \frac{1}{L} \int\limits_{0}^{L} u(x,t)dx \right) + \frac{1}{L} \int\limits_{0}^{L} \mathscr{L}_{\gamma}(u)(x,t) dx = 0,\]
	Therefore
	\[ \frac{d\bar{u}}{dt} + \gamma_0 \bar{u} = 0 .\]
	We deduce the desired result.
\end{proof}

The solution converges to 0 in $H^1$ with a rate depending on the $H_{\gamma}$-norm. Thus, as in \cite{CheSad2}, we introduce the following ratio function
\begin{equation}\label{fctG}
	 G : (u,t) \mapsto G(u,t) = \frac{|u|_{\gamma}}{|u|_{H^1}} = \sqrt{\frac{\sum\limits_{k \in \mathbb{Z}} \gamma_k |\hat{u}_k|^2}{\sum\limits_{k \in \mathbb{Z}} (1+k^2) |\hat{u}_k|^2}} .
\end{equation}
To simplify the writings, we use $G(t)$ instead of $G(u,t)$ when no confusion is possible.  

\begin{Prop}
	Let $u$ be the unique solution of dBBM valued in $\Hg \cap H^1(\T)$. We assume that $G$ is $\mathscr{C}^1$ in time. Then we have the following equalities.
	\begin{enumerate}[i.]
		\item $|u(t)|_{H^1}^2 = e^{-2\int\limits_{0}^{t} G^2(s)ds} |u_0|_{H^1}^2$,
		\item $|u(t)|_{\gamma}^2 = G^2 (t) e^{-2\int\limits_{0}^{t} G^2(s)ds} |u_0|_{H^1}^2$.
	\end{enumerate}
	In particular, $\lim\limits_{t \rightarrow + \infty} |u|_{H^1}^2 = 0$ if and only if $t \mapsto G(t) \notin L_t^2(0,+\infty)$.
\end{Prop}

\begin{proof}
	As previously, we take the scalar product in $L^2(\T)$ of the equation dBBM with $u$ and we obtain
	\[ \frac{1}{2} \frac{d}{dt}|u|_{H^1}^2 + |u|_{\gamma}^2 = 0. \]
	Since $|u|_{\gamma}^2 = G^2(t) |u|_{H^1}^2$, we have
	\[ \frac{1}{2} \frac{d}{dt}|u|_{H^1}^2 + G^2(t) |u|_{H^1}^2 = 0 \mbox{~thus~} |u(t)|_{H^1}^2 = e^{-2\int\limits_{0}^{t} G^2(s)ds} |u_0|_{H^1}^2.\]
	We find $\lim\limits_{t \rightarrow + \infty} |u|_{H^1}^2 = 0$ if and only if $t \mapsto G(t) \notin L_t^2(0,+\infty)$.
	
	Then, we have 
	\[ |u(t)|_{\gamma}^2 = G^2 (t) e^{-2\int\limits_{0}^{t} G^2(s)ds} |u_0|_{H^1}^2. \]
	The proposition \ref{prop4} involves that $\lim\limits_{t \rightarrow + \infty} |u|_{\gamma} = 0$ and consequently
	\[ \lim\limits_{t \rightarrow + \infty} G^2 (t) e^{-2\int\limits_{0}^{t} G^2(s)ds} |u_0|_{H^1}^2 = 0. \]
\end{proof}

\begin{remark}
	We can establish similar results for the forced equation
	\[ u_t - u_{txx} + u_x + uu_x + \mathscr{L}_{\gamma}(u) = f. \]
	\begin{Prop}
		Assume that $f \in H_{1/\gamma}(\T)$ and $u_0 \in H^1(\T) \cap H_{\gamma}(\T)$. Then the solution $u$ of the forced equation satisfies:
		\[ |u(t)|_{H^1}^2 \leq e^{-\int\limits_{0}^{t} G^2(s)ds} |u_0|_{H^1}^2 + \int\limits_{0}^{L} e^{-\int\limits_{s}^{t} G^2(\tau)d\tau} |f|_{1/\gamma}^2 ds. \]	
	\end{Prop}
	
	\begin{proof}
		The scalar product of dBBM with $u$ gives
		\[ \frac{1}{2} \frac{d}{dt}|u|_{H^1}^2 + |u|_{\gamma}^2 = <f,u>. \]
		But
		\[ | \langle f,u \rangle | \leq |f|_{1/\gamma} |u|_{\gamma} \quad \mbox{~and~} \quad |u|_{\gamma}^2 = G^2(t) |u|_{H^1}^2.\]
		From the Young inequality, we have
		\[ \frac{1}{2} \frac{d}{dt}|u|_{H^1}^2 + G^2(t)|u|_{H^1}^2 \leq \frac{1}{2} |f|_{1/\gamma}^2 + \frac{1}{2} |u|_{\gamma}^2. \]
		It involves that 
		\[ \frac{d}{dt}|u|_{H^1}^2 + G^2(t)|u|_{H^1}^2 \leq |f|_{1/\gamma}^2. \]
		Using the Gronwall's lemma, we obtain
		\[ |u(t)|_{H^1}^2(\T) \leq e^{-\int\limits_{0}^{t} G^2(s)ds} |u_0|_{H^1}^2 + \int\limits_{0}^{L} e^{-\int\limits_{s}^{t} G^2(\tau)d\tau} |f|_{1/\gamma}^2 ds. \]	
		Multiplying each term with $G^2(t)$, it gets
		\[ |u(t)|_{\gamma}^2(\T) \leq G^2(t) e^{-\int\limits_{0}^{t} G^2(s)ds} |u_0|_{H^1}^2 + \int\limits_{0}^{L} e^{-\int\limits_{s}^{t} G^2(\tau)d\tau} G^2(t) |f|_{1/\gamma}^2 ds. \]	 
	\end{proof}
\end{remark}

\section{Numerical schemes}
	In this part, we study and present some numerical schemes well suited for the long-time behavior of the solution of dBBM equation. We also give details on their implementation
	\subsection{Numerical schemes}
	Let start with time discretization. We denote by $D$ the operator of derivation in space, by $D^2$ the operator of second derivation in space and by $u^{n}$ the approximation of $u$ at time $t_n=n\Delta t$.

\hspace{0.5cm}

We first introduce the fully bacward Euler method.
We have the following discretization:
\begin{equation}\label{EI}
	(Id-D^2) \frac{u^{n+1}-u^n}{\Delta t} + D u^{n+1} + \frac{1}{2} D \left( (u^{n+1})^2 \right) + \mathscr{L}_{\gamma} u^{n+1}= f.
\end{equation}	

\begin{Prop}\label{prop3_1}
	If $u^0 = u_0 \in H^1(\T)$ and $ f \in H_{1/\gamma}(\T) $ then the sequence $(u^n)_{n \in \N}$ generated by the fully backward Euler method is well defined, in $H^1(\T)$ and 
	\[ |u^{n+1}|_{H^1}^2 + |u^{n+1}-u^n|_{H^1}^{2} + \Delta t \sum\limits_{k \in \mathbb{Z}} \gamma_k |\hat{u}_k^{n+1}|^2 \leq |u^n|_{H^1}^2 + \Delta t \sum\limits_{k \in \mathbb{Z}} \frac{1}{\gamma_k} |\hat{f}_k|^2. \]	
	Moreover, if $f=0$ then $\lim\limits_{n \rightarrow + \infty} |u^n|_{H^1}=0$.
\end{Prop}

\begin{proof}
	We use the following identity
	\[ \langle (Id-D^2)(u^{n+1}-u^n),u^{n+1} \rangle = - \frac{1}{2} \left( |u^n|_{H^1}^2 - |u^{n+1}-u^n|_{H^1}^2  - |u^{n+1}|_{H^1}^2 \right).\]
	We compute the scalar product of (\ref{EI}) with $u^{n+1}$ and we obtain
	\begin{multline*}
		|u^{n+1}|_{H^1}^2 + |u^{n+1}-u^n|_{H^1}^2 - |u^n|_{H^1}^2 + 2\Delta t \langle \mathscr{L}_{\gamma}u^{n+1},u^{n+1} \rangle \\
		= 2 \Delta t \langle u^{n+1},f \rangle, 
	\end{multline*}	
	where
	\[ \langle \mathscr{L}_{\gamma}u^{n+1},u^{n+1} \rangle = \sum\limits_{k\in \mathbb{Z}} \gamma_k |\hat{u}_k^{n+1}|^2 \geq 0.\]
	 Hence we have
	\[ |u^{n+1}-u^n|_{H^1}^2  + |u^{n+1}|_{H^1}^2 + 2\Delta t \sum\limits_{k\in \mathbb{Z}} \gamma_k |\hat{u}_k^{n+1}|^2 = |u^n|_{H^1}^2 + \langle u^{n+1},f \rangle . \]
	The Young inequality on the right hand side provides
	\[ \langle u^{n+1},f \rangle \leq \frac{1}{2} \sum\limits_{k\in \mathbb{Z}} \gamma_k |\hat{u}_k^{n+1}|^2 + \frac{1}{2} \sum\limits_{k\in \mathbb{Z}} \frac{1}{\gamma_k} |\hat{f}_k|^2 \]
	We infer the expected result
	\[ |u^{n+1}|_{H^1}^2 + |u^{n+1}-u^n|_{H^1}^2 + \Delta t |u^{n+1}|_{\gamma}^2 \leq |u^n|_{H^1}^2 + \Delta t \sum\limits_{k\in \mathbb{Z}} \frac{1}{\gamma_k} |\hat{f}_k|^2 \]
	If $f=0$, we directly have the equality
	\[ |u^{n+1}|_{H^1}^2 + |u^{n+1}-u^n|_{H^1}^2 + \Delta t |u^{n+1}|_{\gamma}^2 = |u^n|_{H^1}^2  \]	
	We deduce that the sequence $\left( |u^n|_{H^1} \right)_{n \in \N}$ is decreasing and lower bounded by $0$. Then it is convergent to a positive constant $C$. It follows that
	\[ \lim\limits_{n \rightarrow + \infty} \sum\limits_{k \in \mathbb{Z}} \gamma_k |\hat{u}_k^n|^2 = 0. \]
	As $\gamma_k > 0$, we have $\lim\limits_{n \rightarrow +\infty} \hat{u}_k^n = 0$ and consequently $C=0$.
\end{proof}

\begin{Prop}\label{prop3_2}
	Let $(u^n)_{n \in \N}$ be the sequence generated by the fully backward Euler method. Assume that $f=0$. We set $G^{(n)}=\frac{|u^n|_{\gamma}}{|u^n|_{H^1}}$. Then we have
	\[ |u^n|_{H^1}^2 \leq \left( \prod\limits_{j=1}^{n} \frac{1}{1+2\Delta t (G^{(j)})^2} \right)|u_0|_{H^1}^2. \]
	Moreover, if $\Delta t$ is small enough and if $(G^{(j)})_{j \in \mathbb{Z}} \notin l^2$ then $\lim\limits_{n \rightarrow + \infty} |u^n|_{H^1}=0$.
\end{Prop}

\begin{proof}
	We compute the scalar product of (\ref{EI}) with $u^{n+1}$ and we obtain
	\[ \frac{1}{2\Delta t} \left( |u^{n+1}|_{H^1}^2 - |u^n|_{H^1}^2 + |u^{n+1}-u^n|_{H^1}^2 \right) + (G^{(n+1)})^2 |u^{n+1}|_{H^1}^2 = 0.\]
	Hence
	\[ \left( 1+2\Delta t (G^{(n+1)})^2 \right) |u^{n+1}|_{H^1}^2 + |u^{n+1}-u^n|_{H^1}^2 \leq |u^n|_{H^1}^2. \]
	Therefore
	\[ |u^{n+1}|_{H^1}^2 \leq \frac{|u^n|_{H^1}^2}{ 1+2\Delta t (G^{(n+1)})^2}. \]
	We infer by induction that
	\[ |u^n|_{H^1}^2 \leq \left( \prod\limits_{j=1}^{n} \frac{1}{1+2\Delta t (G^{(j)})^2} \right)|u_0|_{H^1}^2. \]
	Now if $\Delta t$ is small enough, we have
	\[ \log \left( \prod\limits_{j=1}^{n} \frac{1}{1+2\Delta t (G^{(j)})^2} \right) = -\sum\limits_{j=1}^{n}\log(1+2\Delta t (G^{(j)})^2) \simeq -2\Delta t \sum\limits_{j=1}^{n} (G^{(j)})^2.\]
	If $(G^{(j)}) \notin l^2$ then $\lim\limits_{n \rightarrow + \infty} |u^n|_{H^1}=0$.
\end{proof}

\begin{remark}
	The forward Euler method can be applied to the BBM equation, contrarily to the KdV equation. Moreover, it is more difficult to establish qualitative properties as Propositions \ref{prop3_1} and \ref{prop3_2}. The scheme is written as
	\begin{equation*}
		(Id-D^2) \frac{u^{n+1}-u^n}{\Delta t} + D u^{n} + \frac{1}{2} D \left( (u^{n})^2 \right) + \mathscr{L}_{\gamma} u^{n}= f.
	\end{equation*}	
\end{remark}

Let introduce the Sanz-Serna scheme, given by
\begin{multline}\label{SS}
	(Id-D^2)\frac{u^{n+1}-u^n}{\Delta t} + D \left( \frac{u^{n+1}+u^n}{2} \right) + \frac{1}{2} D \left( \left( \frac{u^{n+1} + u^n}{2} \right)^2 \right) \\
	+ \mathscr{L}_{\gamma} \left( \frac{u^{n+1}+u^n}{2} \right) = f.
\end{multline}

\begin{Prop}
	Assume that $u_0 \in H^1(\T)$ and $f \in L^2 \cap H_{1/\gamma}(\T)$. Then the scheme (\ref{SS}) is stable in $H^1(\T)$ for all $\Delta t > 0$.
\end{Prop}

\begin{proof}
	We take the scalar product of (\ref{SS}) with $\frac{u^{n+1}+u^n}{2}$. We obtain
	\[ \frac{|u^{n+1}|_{H^1}^2 - |u^n|_{H^1}^2}{2\Delta t} + \frac{1}{4} |u^{n+1} + u^n|_{\gamma}^2 = \left< f , \frac{u^{n+1}+u^n}{2} \right>. \]
	Using the duality and the Young's inequality, it implies
	\[ \frac{|u^{n+1}|_{H^1}^2 - |u^n|_{H^1}^2}{2\Delta t} + \frac{1}{4} |u^{n+1} + u^n|_{\gamma}^2 \leq \frac{1}{2} |f|_{1/\gamma}^2 + \frac{1}{2} \left| \frac{u^{n+1}+u^n}{2} \right|_{\gamma}^2. \]
	Hence
	\[ |u^{n+1}|_{H^1}^2 + \frac{\Delta t}{4} |u^{n+1} + u^n|_{\gamma}^2 \leq |u^n|_{H^1}^2 + \Delta t |f|_{1/\gamma}^2.\]
	Consequently we have the stability on every time interval $[0,T]$.
\end{proof}

\begin{remark}
	The Crank-Nicolson scheme is given by 
\begin{multline}\label{CN}
	(Id-D^2)\frac{u^{n+1}-u^n}{\delta t} + D \left( \frac{u^{n+1}+u^n}{2} \right) + \frac{1}{4} D \left( (u^{n+1})^2 + (u^n)^2 \right) \\
	+ \mathscr{L}_{\gamma} \frac{u^{n+1}+u^n}{2} = f.	
\end{multline}
	We cannot have uniform $H^1-$bounds for $u^n$ from this scheme because
\[ \left< D \left( (u^{n+1})^2 + (u^n)^2 \right) , \frac{u^{n+1}+u^n}{2} \right> \neq 0. \]
\end{remark}

	\subsection{Implementation}
	The Fourier Transform in space is applied to the dBBM equation. The equation \eqref{eqBBM2} becomes
\[ (1+k^2) \frac{d\hat{u}_k}{dt} + (ik + \gamma_k) \hat{u}_k + (\widehat{u u_x})_k = 0. \]
We remind that the term $\mathscr{L}_{\gamma}(u)$ is written in the Fourier space as \[\sum\limits_{k \in \Z} \gamma_k \hat{u}_k e^{i kx}.\]

Then we use the schemes introduced in the previous subsection for the discretization in time and the FFT for the space discretization. These are implicit methods and we apply a fixed-point method. We solve equations which can be written as $u^{n+1}=\Phi(u^n,u^{n+1})$ by Picard iterates.

\begin{algorithm}[H]
\caption{Picard iterates}
\begin{algorithmic}
\STATE $v^0=u^n$, $m=0$
\WHILE{$\| \Phi(u^n,u^m)-u^m \| > \epsilon$} 
	\STATE $v^{m+1}=\Phi(u^n,v^m)$
	\STATE $m=m+1$ 
\ENDWHILE
\STATE $u^{n+1}=v^m$
\end{algorithmic}
\end{algorithm}

Unfortunately, even if the used scheme is unconditionally stable, this method presents instabilities. The stability of the fixed-point iterate can be improved by applying extrapolation like technic as follows (see also \cite{Abo}). The principle is to change the iterate $v^{m+1}=\Phi(u^n,v^m)$ by
\[ v^{m+1}=v^n - (-1)^k \alpha_{m}^{k} \Delta_{\Phi}^{k} v^{m},\]
where $\Delta_{\Phi}^{k} v^{m} = \sum\limits_{j=0}^{k}C_{j}^{k}(-1)^{k-j} \Phi^{(j)}(u^n,v^m)$, $\Phi^{(j)}$ denotes the $j$-th composition of $\Phi$ with itself.
The parameter $\alpha_m^k$ is computed such as minimizing the Euclidean norm of the linearized part of the residual $r^{m+1}=v^{m+1}-\Phi(u^n,v^m)$. We have 
\[ \alpha_m^k = (-1)^k \frac{ \langle \Delta_{\Phi}^1 v^m , \Delta_{\Phi}^{k+1} v^m \rangle}{\langle \Delta_{\Phi}^{k+1} v^m , \Delta_{\Phi}^{k+1} v^m \rangle }. \]
In fact, we take $k=1$, and we have the following algorithm                                                                                                                                                                                                                                                                                                                     

\begin{algorithm}[H]
\caption{Improved fixed-point algorithm}
\begin{algorithmic}
\STATE $v^0=u^n$, $m=0$
\WHILE{$\| \Phi(u^n,v^m)-v^m \| > \epsilon$} 
	\STATE $\Delta_{\Phi}^1 v^m = \Phi(u^n,v^m)-v^m$
	\STATE $\Delta_{\Phi}^2 v^m = \Phi(u^n,\phi(u^n,v^m)) - 2\Phi(u^n,v^m) + v^m$
	\STATE $ \alpha_m^1 = - \frac{\langle\Delta_{\Phi}^1 v^m , \Delta_{\Phi}^{2} v^m \rangle}{\langle\Delta_{\Phi}^{2} v^m , \Delta_{\Phi}^{2} v^m \rangle} $
	\STATE $v^{m+1}=v^m + \alpha_m^1 \Delta_{\Phi}^{1} v^m$
	\STATE $m=m+1$ 
\ENDWHILE
\STATE $u^{n+1}=v^m$
\end{algorithmic}
\end{algorithm}

\section{Numerical results}
	\graphicspath{{pictures/}}

%-----------------------------------------------------------------------------------%
%                          AMORTISSEMENT LAPLACIEN                                  %
%-----------------------------------------------------------------------------------%

Let us begin by stating the parameters chosen for the simulations. We present the results obtained with the Sanz-Serna scheme for the non-forced equation ($f=0$) on the interval $[-L,L]$, with $L=50$, discretized in $2^{10}$ points. We obtain similar results with fully backaward Euler and Crank-Nicolson schemes. The time step is $\Delta t=0.01$.
Let us remind that $k$ can denote either an integer as an index, or the value $\frac{2\pi k}{L}$.  
%We remind that $k$ is used for $\frac{2\pi k}{L}$ if $k$ is not an index. Actually in this section, $k$ is an integer only when it is the index of $\gamma_k$.
\begin{figure}[H]
	\begin{changemargin}{-1.4cm}{0cm}
	\centering
	\includegraphics[scale=0.29]{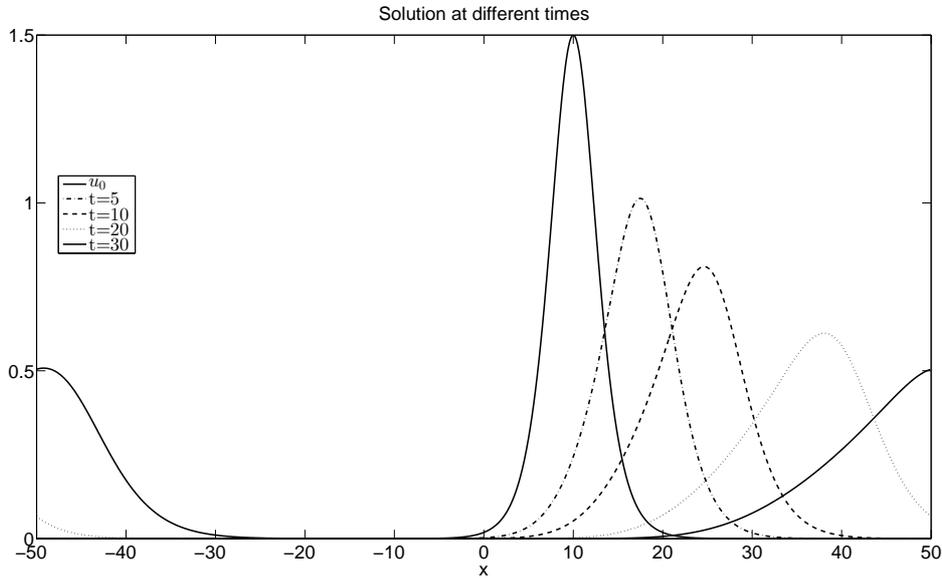} 
	\end{changemargin}	
	\caption{Solution at times $t=0$ (---), $t=5$ ($\cdot-\cdot$), $t=10$ (- - -), $t=20$ ($\cdots$), $t=30$ (---) for $\gamma_k = k^2$. Here the initial datum is the soliton.}
	\label{fig_1}
\end{figure}
We consider the soliton as initial datum
\[ u_0(x)=-3(1-c)\cosh^{-2}\left( \pm \frac{1}{2}\sqrt{\frac{c-1}{c}}(x-d) \right), \]
with $c=1.5$ and $d=0.2L$ .
We begin with a laplacian damping in Figure \ref{fig_1}. The figure shows the effect of the damping on the solution. Since $\gamma_0 = 0$, the solution tends to the mean value of $u_0$.

%-----------------------------------------------------------------------------------%
%                          AMORTISSEMMENTS CLASSIQUES                               %
%-----------------------------------------------------------------------------------%

Figure \ref{fig_2} corresponds to the simulation of the following standard dampings
\begin{itemize}
	\item the constant: $\gamma_k=1, \ \forall k$,
	\item the laplacian: $\gamma_k=k^2, \ \forall k$,
	\item the bilaplacian: $\gamma_k=k^4, \ \forall k$. 
\end{itemize}
\begin{figure}[H]
	\begin{changemargin}{-1.4cm}{0cm}
	\centering
	\includegraphics[scale=0.29]{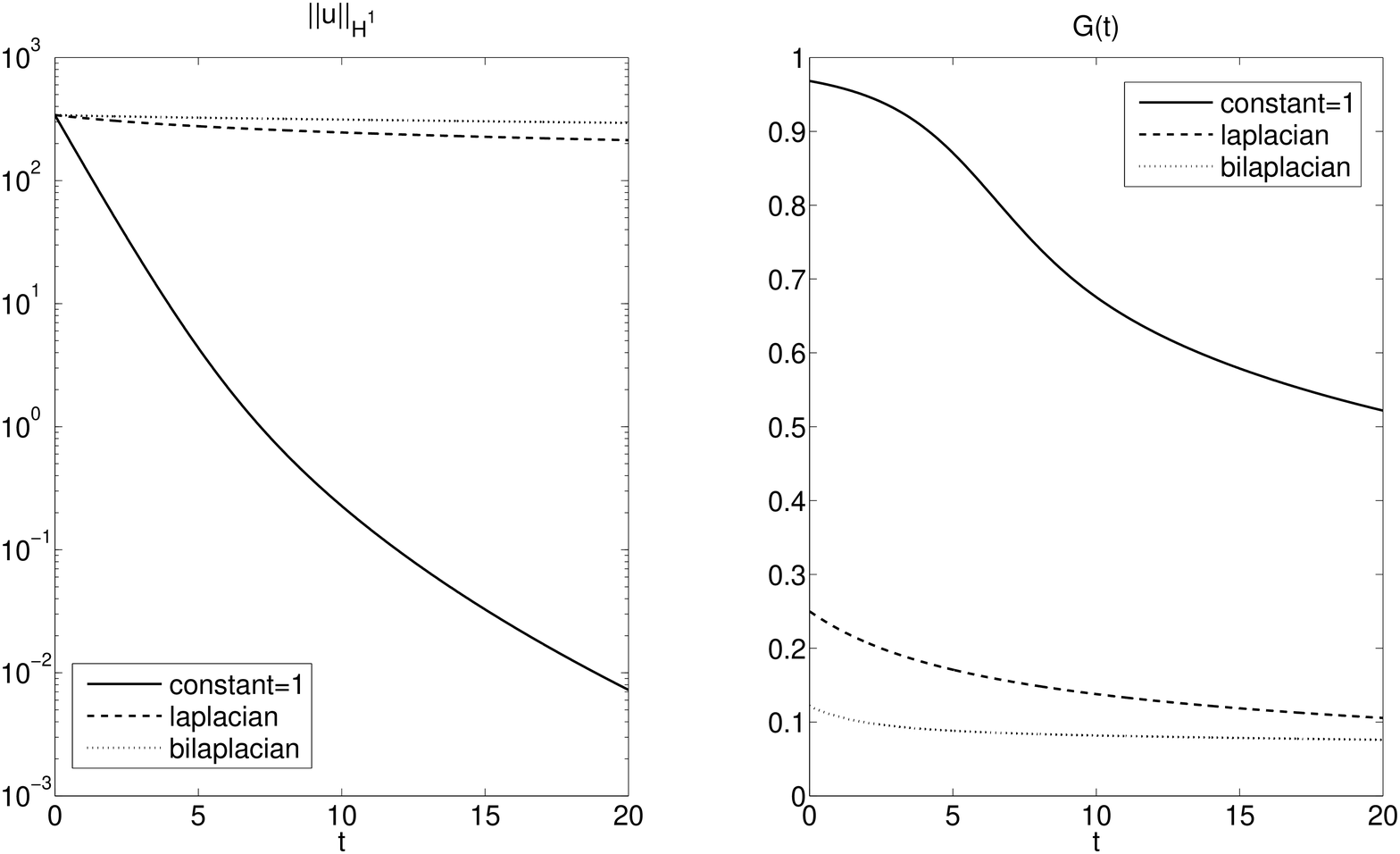} 
	\end{changemargin}                                                                                                                                                                                                                                                                                                                                                                                                                                                                                                                                                                                                                                                                                                                                                                                                                                                                                                                                                                                                                                                                                                                                                                                                                                                                                                                                                                                                                                                                                                                                                                                                                                                                                                                                                                                                                                                                                                                                                                                                                                                                                                                                                                                                                                                                                                                                                         
	\caption{Evolution of $\|u\|_{H^1}$ and $G(t)$ with respect to time for standard dampings, the constant (---), the laplacian (- - -) and the bilaplacian ($\cdots$). Here the initial datum is the soliton.}
	\label{fig_2}
\end{figure}
For these dampings, we verify that the $H^1$-norm decreases as expected. Moreover the solution converge to the norm value of $u_0$ for the laplacian and bilpalacian dampings. We can also classify these dampings. Here the constant one is more efficient than the two other. But the laplacian one is more efficient than the bilaplacian one. These results are in agreement with the ones about the KdV equation \cite{CheSad1, Ghi1, Ghi2, Gou, GouRos, OttSud}.

%-----------------------------------------------------------------------------------%
%                          AMORTISSEMENTS TENDANT VERS 0                            %
%-----------------------------------------------------------------------------------%
We extend, in Figure \ref{fig_3}, the simulations to the dampings such that $\gamma_k > 0$ with $\lim\limits_{k \rightarrow +\infty} \gamma_k = 0$.
We consider here three "tend to 0" dampings, with polynomial or exponential decay as follows
\begin{itemize}
	\item $\gamma_k  = \frac{1}{1+|k|}$,
	\item $\gamma_k  = \frac{1}{(1+|k|)^{\alpha}}$, $\alpha=3$
	\item $\gamma_k  = e^{-|k|}$.
\end{itemize}
These dampings do not fit with the asymptions of the theorem \ref{thm_local} but they still damp down the soliton. Indeed we notice that the $H^1$-norm is decreasing. We can classify the dampings. Precisely, the first one damps more than the two other and the third one is more efficient than the second one. Since these dampings tends to 0, their effects are principally on low frequencies. And for low frequencies, we have
\[ \frac{1}{1+|k|} > e^{-|k|} > \frac{1}{(1+|k|)^{\alpha}}. \]
Then the classification of dampings in Figure \ref{fig_3} seems licit.

\begin{figure}[H]
	\begin{changemargin}{-1.4cm}{0cm}
	\centering
	\includegraphics[scale=0.29]{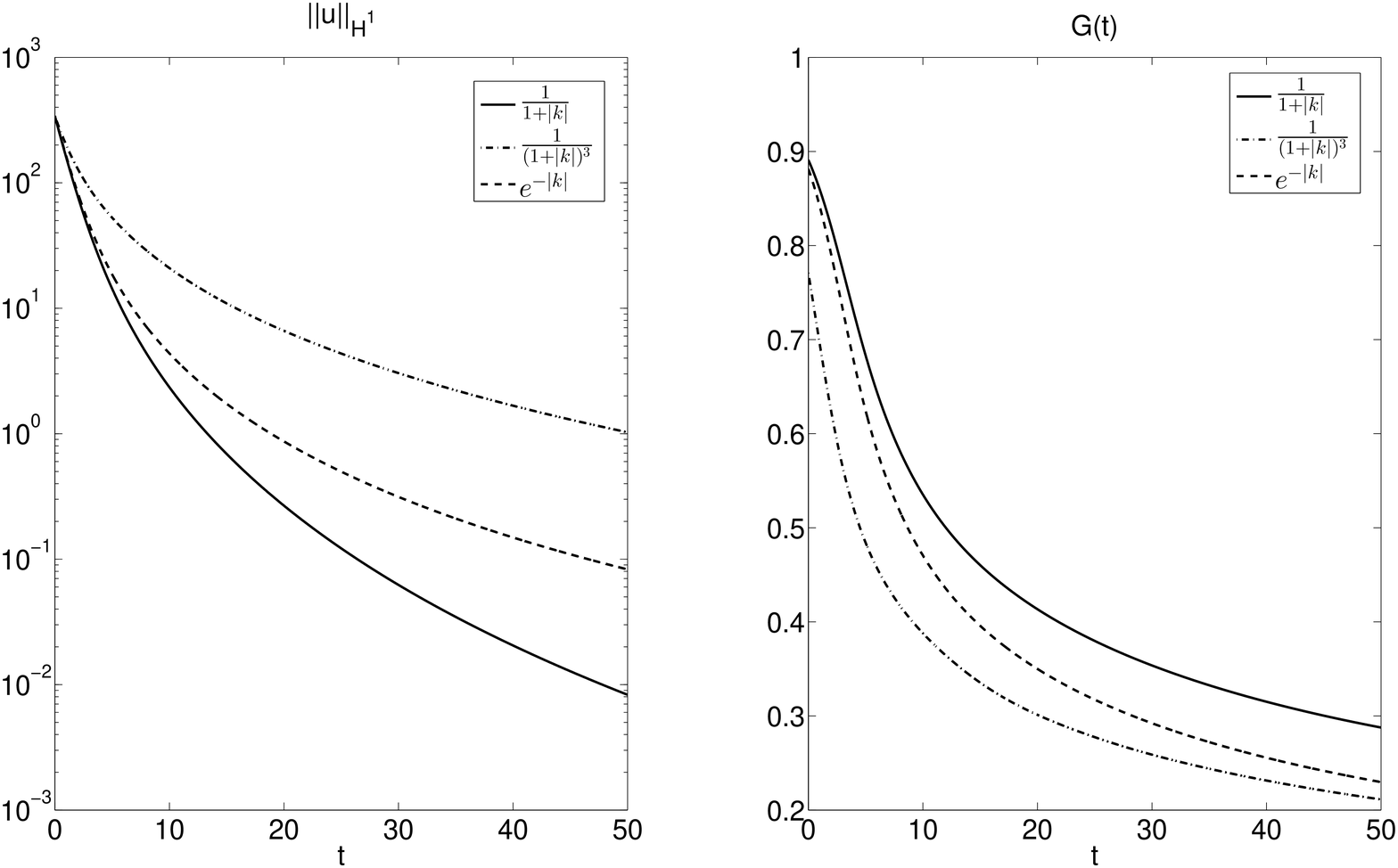} 
	\end{changemargin}
	\caption{Evolution of $\|u\|_{H^1}$ and $G(t)$ with respect to time for "tend to 0" dampings, $\gamma_k  = \frac{1}{1+|k|}$ (---), $\gamma_k  = \frac{1}{(1+|k|)^{3}}$ ($\cdot-\cdot$) and $\gamma_k  = e^{-|k|}$ (- - -). Here the initial datum is the soliton.}
	\label{fig_3}
\end{figure}

%-----------------------------------------------------------------------------------%
%                          AMORTISSEMENTS INDICATRICE                               %
%-----------------------------------------------------------------------------------%

We now consider band limited dampings, namely
\begin{equation}\label{eqI}\tag{$I(N)$}
	\gamma_k = \left\{ \begin{array}{l} 1 \mbox{~if~} |k| \leq N \\ 0 \mbox{~if~} |k| > N \end{array} \right. .
\end{equation}	
We define $M := \max (\textbf{k})$, where $\textbf{k}$ is the finite vector of frequencies used for the numerical tests in Figure \ref{fig_4}. 
%The two have similar damping at the beginning but the one with largest scale of frequencies damps more longer.
To keep $\gamma_k >0, \ \forall k$, we choose 
\begin{equation}\label{eqE}\tag{$EI(N)$}
	\gamma_k = \left\{ \begin{array}{l} 1 \mbox{~if~} |k| \leq N \\ e^{-a ||k|-N|} \mbox{~if~} |k| > N \end{array} \right. .
\end{equation}	
or
\begin{equation}\label{eqP}\tag{$PI(N)$}
	\gamma_k = \left\{ \begin{array}{l} 1 \mbox{~if~} |k| \leq N \\ \frac{1}{\left( 1+||k|-N| \right)^a} \mbox{~if~} |k| > N \end{array} \right. .
\end{equation}	
We denote the classic band limited damping by $I(N)$, the one corrected with exponential decay $EI(N)$ and the third one corrected with polynomial decay $PI(N)$. Here we set $a=3$, we obtain $\gamma_k > 0$ with $\lim\limits_{k \rightarrow +\infty} \gamma_k = 0$.

\begin{figure}[H]
	\begin{changemargin}{-1.4cm}{0cm}
	\centering
	\includegraphics[scale=0.29]{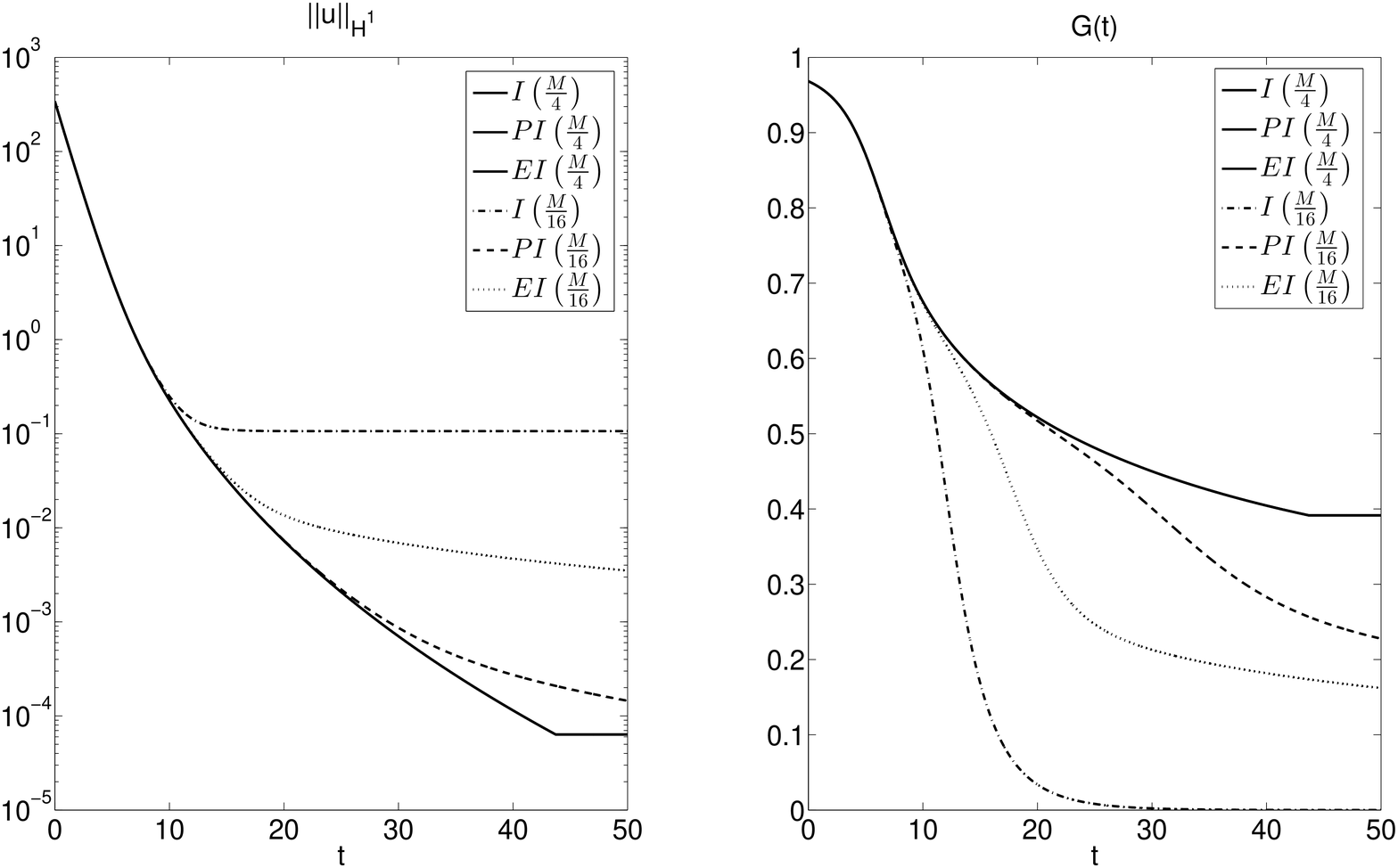} 
	\end{changemargin}
	\caption{Evolution of $\|u\|_{H^1}$ and $G(t)$ with respect to time for band limited type dampings, $I \left( \frac{M}{4} \right)$ (---), $PI \left( \frac{M}{4} \right)$ (---), $EI \left( \frac{M}{4} \right)$ (---), $I \left( \frac{M}{16} \right)$ ($\cdot-\cdot$), $PI \left( \frac{M}{16} \right)$ (- - -) ,$EI \left( \frac{M}{16} \right)$ ($\cdots$). Here the initial datum is the soliton.}
	\label{fig_4}
\end{figure}

The difference between these dampings is when $|k|>N$. For these frequencies, since $\gamma_k \neq 0$, the dampings with exponential or polynomial decay should damp more than the ideal one ($I(N)$). It is the case when $N=\frac{M}{16}$, moreover the damping with polynomial decay is more effective than the one with exponential decay. But we observe similar damping when $N=\frac{N}{4}$. We notice the damping depends on the band width. It seems rightful because the equation considered is dispersive. Indeed, for this kind of equation, high frequencies appear and these frequencies cannot be damped especially if the band width is too small.

Let us illustate this by considering now the gaussian as initial datum
\[ u_0(x) = e^{-\frac{x^2}{\sigma^2}}. \]
Here the standard deviation is fixed equal to $\sigma=\frac{\max(\textbf{k})}{10}$. Then, in order to see the effect of the damping on the low frequencies, we study the influence of the band limited damping size in Figure \ref{fig_5}.
We observe that the more $N$ increases, the more efficient is the damping. However, it seems that all dampings with $N \geq \frac{3\sigma}{2}$ behave similarly. Besides, for lower $N$, the different kind of dampings can be classified. We notice that for all $N$ and all dampings, the decreasing speed is similar at beginning and the difference appears after. The difference appears when the high frequencies does. 

\begin{figure}[H]
	\begin{changemargin}{-1.4cm}{0cm}
	\centering
	\includegraphics[scale=0.29]{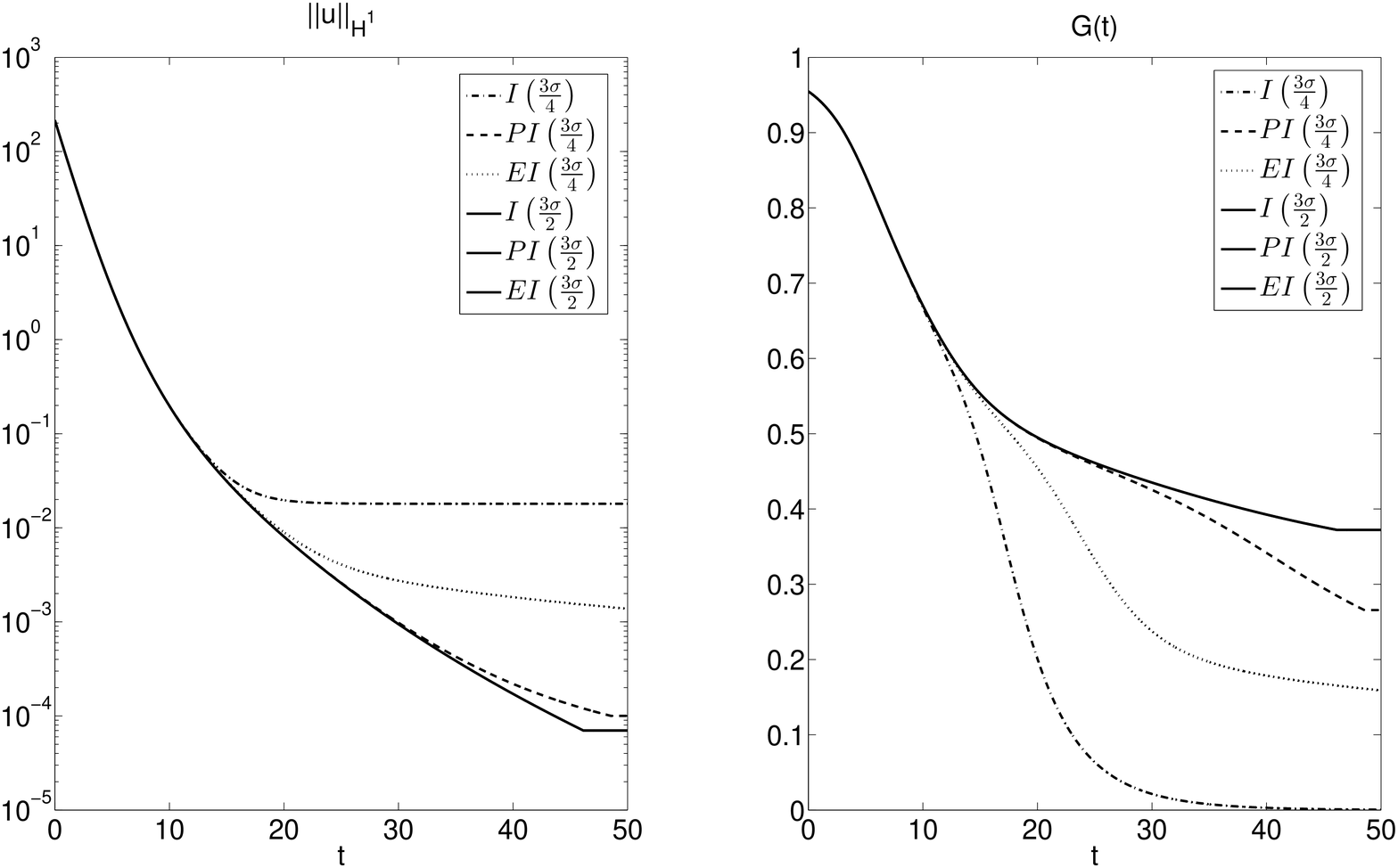} 
	\end{changemargin}
	\caption{Evolution of $\|u\|_{H^1}$ and $G(t)$ with respect to time for band limited type dampings, $I \left( \frac{3\sigma}{4} \right)$ ($\cdot-\cdot$), $PI \left( \frac{3\sigma}{4} \right)$ (- - -), $EI \left( \frac{3\sigma}{4} \right)$ ($\cdots$), $I \left( \frac{3\sigma}{2} \right)$ (---), $PI \left( \frac{3\sigma}{2} \right)$ (---) ,$EI \left( \frac{3\sigma}{2} \right)$ (---). Here the initial datum is the gaussian.}
	\label{fig_5}
\end{figure}

%-----------------------------------------------------------------------------------%
%                          INFLUENCE LONGUEUR DOMAINE                               %
%-----------------------------------------------------------------------------------%

The influence of the domain is finally inspected. It is known that the eigenvalues of the laplacian and of the bilapalcian depends on the length of the domain $L$. Here we choose the initial datum as 
\[ u_0(x) = \sin \left( \frac{2\pi x}{L} \right). \] 
Figure \ref{fig_6} presents results with $L$ equal to $\pi$, $2\pi$ and $3\pi$.
We notice that  if $L < 2\pi$, the laplacian damping is less effective than the bilaplacian one. It is the contrary if $L > 2\pi$. But, the dampings are similar when $L=2\pi$. We also observe that for a fixed damping, the greater is $L$, the slower is the damping.

\begin{figure}[H]
	\begin{changemargin}{-1.4cm}{0cm}
	\centering
	\includegraphics[scale=0.29]{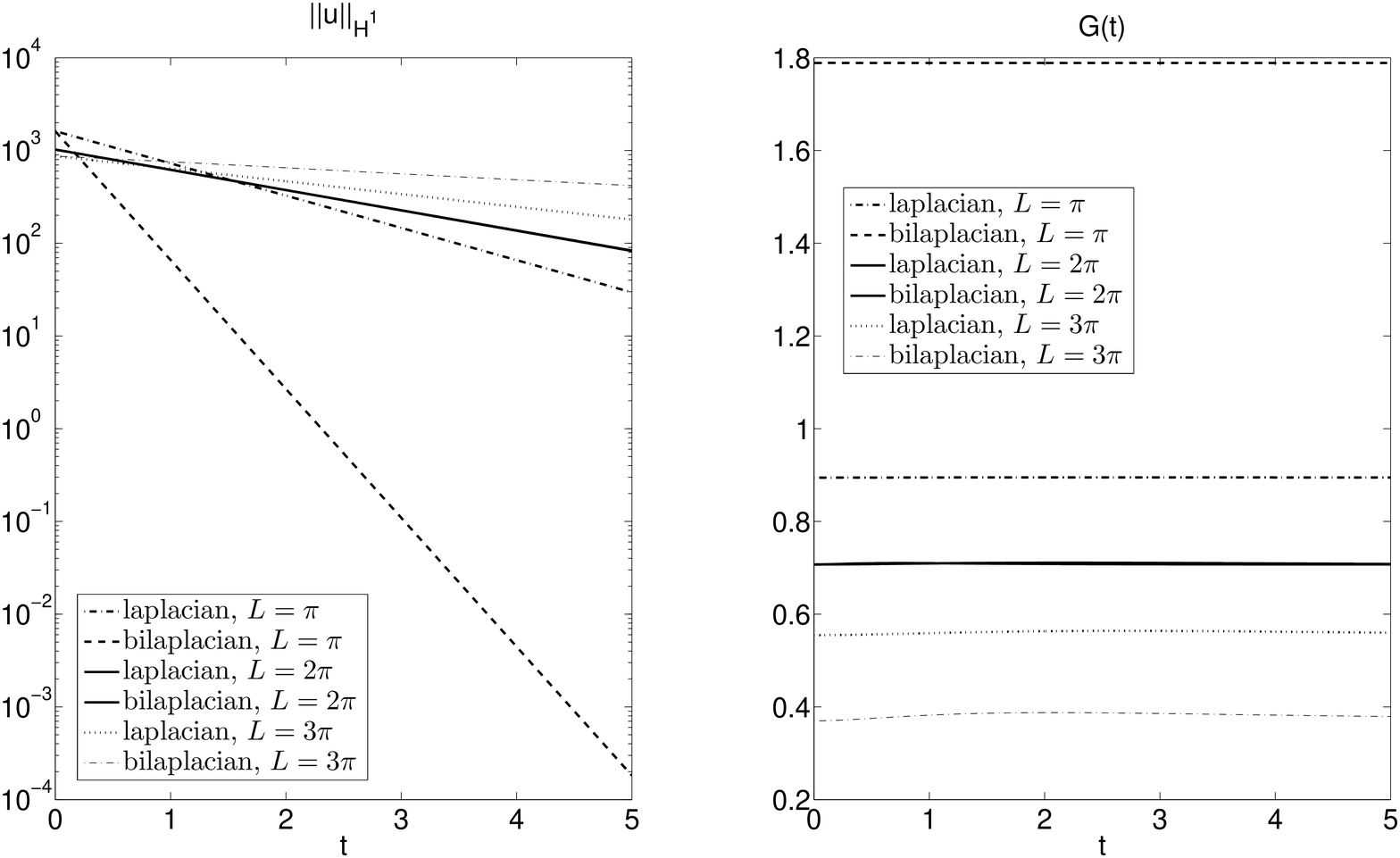} 
	\end{changemargin}
	\caption{Evolution of $\|u\|_{H^1}$ and $G(t)$ with respect to time for $L=\pi$ with $\gamma_k = k^2$ ($\bm{\cdot-\cdot}$) and $\gamma_k = k^4$ (- - -), $L=2\pi$ with $\gamma_k = k^2$ (---) and $\gamma_k = k^4$ (---), $L=3\pi$ with $\gamma_k = k^2$ ($\cdots$) and $\gamma_k = k^4$ ($\cdot-\cdot$). Here the initial datum is a sine.}
	\label{fig_6}
\end{figure}

\section{Final comments}
	The damping operator $\mathscr{L}_{\gamma}$ considered in this article allows to have a large variety of dampings. We first get back the standard dampings (for example $-\Delta u$). But we also have less common dampings like band-limited ones or weaker ones, e.g., such that $\lim\limits_{|k| \rightarrow +\infty} \gamma_k = 0$.
Moreover, with this frequential approach, we can adjust the dampings by frequency bands. This is interesting in order to build cheap and efficient dampings.
Using these dampings, the numerical observation shows a damping for energy norms like the $H^1$-norm. This is consistent with the results about the asymptotic behavior.

This way of damping with frequency filters is particulary of interest for its flexibility and its efficiency. It would be a good perspective to perform similar work for the control wave equations.
	
%\nocite{*}
%\bibliographystyle{unsrt}
%\bibliography{biblio}

\end{document}